\documentclass[a4paper,12pt]{amsart}

\newcommand{\cMqui}{\operatorname{\scM_{\mathrm{qui}}}}

\newcommand{\Mquibar}{\operatorname{\Mbar_{\mathrm{qui}}}}

\newcommand{\Mquadbar}{\operatorname{\Mbar_{\mathrm{quad}}}}

\newcommand{\cMpot}{\operatorname{\scM_{\mathrm{pot}}}}

\newcommand{\Mpotbar}{\operatorname{\Mbar_{\mathrm{pot}}}}

\newcommand{\Trace}{\operatorname{Tr}}

\usepackage{mymacros}
\usepackage{xypic}

\newcommand{\bk}{\bsk}

\newcommand{\op}{\mathrm{op}}

\newcommand{\cycl}{\mathsf{cycl}}

\newcommand{\vir}{\mathrm{vir}}

\title{Noncommutative quadric surfaces and
noncommutative conifolds}
\author{Shinnosuke Okawa, Kazushi Ueda}
\date{}
\dedicatory{To the memory of Kentaro Nagao}
\begin{document}

\maketitle

\begin{abstract}
We introduce a compact moduli of noncommutative quadrics,
and show that it is the weighted projective space
of weight (2,4,4,6).
We also introduce a compact moduli of potentials
for the conifold quiver,
and show that it is the weighted projective space
of weight (1,2,3,4).
There is a natural morphism
from the latter to the former,
which is finite of degree 4.
\end{abstract}


\section{Introduction}

The \emph{conifold} is an ordinary double point in dimension 3.
It is a cone over $\bP^1 \times \bP^1$ from the complex point of view,
and a cone over $S^2 \times S^3$ from the real point of view.
It is the most basic singularity in dimension 3,
which appears in various areas of geometry and physics,
such as Reid's fantasy \cite{MR848512,MR909231},
conifold transition \cite{Batyrev-Ciocan-Fontanine-Kim-van_Straten_CT},
Sasaki-Einstein geometry \cite{MR574272},
knot theory \cite{MR1765411}, and
AdS/CFT correspondence \cite{MR1666725}.

As a cone over $\bP^1 \times \bP^1$,
the conifold $X$ is obtained by contracting the zero-section
of the total space
$
 \Ytilde = \cSpec _{\bP^1 \times \bP^1}\lb
   \bigoplus_{n=0}^\infty \cO_{\bP^1 \times \bP^1}(1,1)^{\otimes n}
    \rb
$
of a line bundle on $\bP^1 \times \bP^1$.
The contractions along two rulings of $\bP^1 \times \bP^1$
give two crepant resolutions of the conifold;
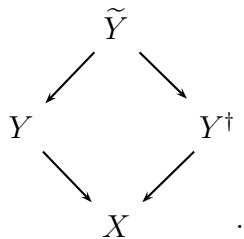
\begin{align*}
\begin{psmatrix}[colsep=9mm,rowsep=7mm]
 & \Ytilde \\
 Y & & Y^\dagger \\
 & X
\end{psmatrix}.
\psset{nodesep=5pt,arrows=->}
\ncline{1,2}{2,1}
\ncline{1,2}{2,3}
\ncline{2,1}{3,2}
\ncline{2,3}{3,2}
\end{align*}
Both $Y$ and $Y^\dagger$ are isomorphic to
$
 \cSpec_{\bP^1} \lb \cSym^* \lb \cO_{\bP^1}(1) \oplus \cO_{\bP^1}(1) \rb \rb.
$
The birational morphism
$
 Y \dashrightarrow Y^\dagger
$
obtained from the above diagram,
known as the \emph{Atiyah flop}
\cite{MR0095974},
is responsible for the non-separatedness of the moduli space
of K3 surfaces.

Besides two crepant resolutions,
the conifold has a \emph{noncommutative crepant resolution}
\cite{Van_den_Bergh_TFNR,Van_den_Bergh_NCR}.
It is an algebra $A_0$ of the form $\End_R M$,
where $R_0=\bC[x,y,z,w]/(xy-zw)$ is the coordinate ring of $X$, and
$M=R_0 \oplus I$, $I=(x,z) \subset R_0$, is a reflexive $R_0$-module.
All these crepant resolutions are derived-equivalent
\cite{Bondal-Orlov_semiorthogonal,Bridgeland_FDC,Van_den_Bergh_TFNR};
\begin{align}
 D(\Qcoh Y) \cong D(\Module A_0) \cong D(\Qcoh Y^\dagger).
\end{align}

\begin{figure}[t]
\begin{minipage}{.45 \linewidth}
\centering
\vspace{10mm}
$
\begin{psmatrix}[mnode=circle]
 v_0 & & v_1
\psset{nodesep=3pt,arrows=->}
\ncarc[arcangle=17]{1,1}{1,3}
\ncarc[arcangle=30]{1,1}{1,3}^{a_1, a_2}
\ncarc[arcangle=17]{1,3}{1,1}
\ncarc[arcangle=30]{1,3}{1,1}_{b_1, b_2}
\end{psmatrix} 
$
\vspace{13mm}\\
\caption{The conifold quiver}
\label{fg:conifold_quiver}
\end{minipage}
\begin{minipage}{.45 \linewidth}
\centering
\vspace{8mm}
$
\begin{psmatrix}[mnode=circle]
v_{1,1} & & v_{0,1} \\
v_{0,0} & & v_{1,0}
\psset{nodesep=3pt,arrows=->,shortput=tablr}
\ncline[offset=2pt]{2,1}{2,3}^{a_1}
\ncline[offset=-2pt]{2,1}{2,3}_{a_2}
\ncline[offset=2pt]{2,3}{1,3}<{b_1}
\ncline[offset=-2pt]{2,3}{1,3}>{b_2}
\ncline[offset=-2pt]{1,3}{1,1}^{a_1'}
\ncline[offset=2pt]{1,3}{1,1}_{a_2'}
\ncline[offset=-2pt]{1,1}{2,1}<{b_1'}
\ncline[offset=2pt]{1,1}{2,1}>{b_2'}
\end{psmatrix} 
$
\vspace{0mm}
\caption{The double cover}
\label{fg:double_quiver}
\end{minipage}
\begin{minipage}{.45 \linewidth}
\centering
\vspace{13mm}
$
\begin{psmatrix}[mnode=circle,colsep=20mm]
 v_{0,0} & v_{1,0} & v_{0,1} & v_{1,1}
\psset{nodesep=3pt,arrows=->}
\ncline[offset=2pt]{1,1}{1,2}^{a_1}
\ncline[offset=-2pt]{1,1}{1,2}_{a_2}
\ncline[offset=2pt]{1,2}{1,3}^{b_1}
\ncline[offset=-2pt]{1,2}{1,3}_{b_2}
\ncline[offset=2pt]{1,3}{1,4}^{a_1'}
\ncline[offset=-2pt]{1,3}{1,4}_{a_2'}
\end{psmatrix} 
$
\vspace{3mm}\\
\caption{The $\bP^1 \times \bP^1$ quiver}
\label{fg:P1P1_quiver}
\end{minipage}
\end{figure}

The algebra $A_0$ can be descibed as the Jacobi algebra
$\bC Q/(\partial \Phi_0)$
of the quiver in \pref{fg:conifold_quiver}
with the potential
\begin{align} \label{eq:conifold_potential}
 \Phi_0 = a_1 b_1 a_2 b_2 - a_1 b_2 a_2 b_1.
\end{align}
Szendr\H{o}i \cite{Szendroi_NCDT} introduced the notion
of \emph{noncommutative Donaldson invariant},
and conjectured that its generating function
\begin{align} \label{eq:DT}
 Z(q) = \sum_{v \in \bZ^2} D(v) q^v, \qquad
 D(v) = \sum_{n \in \bZ} n \chi(\nu^{-1}(n)),
\end{align}
has an infinite product expansion
\begin{align}
 Z(q) = Z_{\mathrm{DT}}(Y; q_0 q_1, q_1^{-1})
  \prod_{m=1}^\infty \lb 1+q_0^m(-q_1)^{m+1} \rb^m.
\end{align}
Here $\nu : \cM(v) \to \bZ$ is the Behrend function
\cite{MR2600874}
on the moduli space $\cM(v)$ of cyclic $\Atilde$-modules
of dimension vector $(v_0, v_1, 1)$, and
 $Z_\mathrm{DT}(Y; q_0 q_1, q_1^{-1})$
is the \emph{Donaldson-Thomas invariant}
\cite{MR1634503,MR1818182,
Maulik-Nekrasov-Okounkov-Pandharipande_I,
Maulik-Nekrasov-Okounkov-Pandharipande_I}
of the resolved conifold $Y$.
Szendr\H{o}i's conjecture is proved by Young
\cite{Young_CPP} in a combinatorial way,
and by Nagao and Nakajima \cite{MR2836398}
in an algebro-geometric way.
This result is generalized to small toric Calabi-Yau 3-folds
in \cite{MR2999994, MR2643058},
and to motivic invariants in \cite{MR2927365,Morrison-Nagao}.

In this paper,
we consider deformations of the algebra $A_0$
as a \emph{graded Calabi-Yau algebra of dimension 3}.
Since any graded Calabi-Yau algebra of dimension 3
comes from a quiver with potential
\cite[Theorem 3.1]{Bocklandt_GCYA},
we study deformation of potentials
instead of deformation of algebras.
The potential of the conifold quiver
in \pref{fg:conifold_quiver}
is an element of
$
 \Sym^2 (V_0 \otimes V_1),
$
where
$
 V_0
$
and
$
 V_1
$
are two-dimensional vector spaces,
and two potentials give isomorphic graded algebras
if they are related by the action of $\GL(V_0) \times \GL(V_1)$.
This motivates us to define the moduli stack of potentials as
\begin{align}
 \cMpot = [\Sym^2(V_0 \otimes V_1) / \GL(V_0) \times \GL(V_1)].
\end{align}
The corresponding compact moduli scheme,
obtained as the GIT quotient,
will be denoted by
\begin{align}
 \Mpotbar = \Proj \lb
  \bC \ld \Sym^2(V_0 \otimes V_1) \rd^{\SL(V_0) \times \SL(V_1)}
    \rb.
\end{align}

The double cover of the conifold quiver
shown in \pref{fg:double_quiver}
inherits the potential
from the conifold quiver.
When the potential for the conifold quiver
is the classical one in \eqref{eq:conifold_potential},
then the derived category of modules
over the Jacobi algebra of the double cover
is equivalent to the derived category of coherent sheaves
on the total space of the canonical bundle of $\bP^1 \times \bP^1$.
By taking the quotient with respect to the ideal $\la b_1', b_2' \ra$
generated by arrows $b_1'$ and $b_2'$,
one obtains the quiver in \pref{fg:P1P1_quiver},
equipped with two relations among eight paths
from the source vertex $v_{0,0}$ to the sink vertex $v_{1,1}$.
This quiver with relations is derived-equivalent to $\bP^1 \times \bP^1$.
This suggests that
deformations of the relations
of the quiver \pref{fg:P1P1_quiver}
describe noncommutative $\bP^1 \times \bP^1$.
This indeed is the case,
and the resulting noncommutative deformation of $\bP^1 \times \bP^1$
is a noncommutative quadric surface
in the sense of Van den Bergh
\cite{MR2836401}.

According to Bondal and Polishchuk \cite{MR1230966}
and Van den Bergh \cite{MR2836401},
noncommutative quadric surfaces are classified
either by
\begin{itemize}
 \item
a quintuple $(V_0, V_1, V_2, V_3, W)$
of 2-dimensional vector spaces $V_0$, $V_1$, $V_2$, $V_3$
and a one-dimensional subspace
$W \subset V_0 \otimes V_1 \otimes V_2 \otimes V_3$, or
 \item
a quadruple $(E, \scrL_0, \scrL_1, \scrL_2)$
of a projective curve $E$ of arithmetic genus one and
three line bundles $\scrL_0$, $\scrL_1$, $\scrL_2$ on $E$
satisfying certain conditions.
\end{itemize}
The moduli stack of quintuples is given by
\begin{align}
 \cMqui = [V_0 \otimes V_1 \otimes V_2 \otimes V_3 \otimes W^{\vee}
  / \GL(V_0) \times \GL(V_1) \times \GL(V_2) \times \GL(V_3) \times \GL(W)],
\end{align}
and the corresponding GIT qutient is given by
\begin{align} \label{eq:Mquibar}
 \Mquibar = \Proj \lb
   \bC \ld
    V_0 \otimes V_1 \otimes V_2 \otimes V_3
   \rd^{\SL(V_0) \times \SL(V_1) \times \SL(V_2) \times \SL(V_3)}
  \rb.
\end{align}
As for quadruples,
we give a birational parametrization
in terms of a quotient of the fiber product
of two copies of certain elliptic surface
over the modular curve;
\begin{align}
 \Mquadbar = \left. S(2) \times_{X(2)} S(2)
  \right/ \lb \SL_2(\bF_2) \ltimes (\bZ/2\bZ)^4 \rb.
\end{align}
The main result in this paper is the following:

\begin{theorem} \label{th:main}
\ \\[-5mm]
\begin{enumerate}
 \item \label{it:quintuple}
The moduli stack $\cMqui$,
after removing the generic stabilizer group $\bG_m^4$,
is birational to the GIT quotient $\Mquibar$.
This in turn is isomorphic
to the weighted projective space
$\bP(2,4,4,6)$.
 \item \label{it:quadruple}
The moduli spaces $\Mquibar$ and $\Mquadbar$ are birational.
 \item \label{it:potential}
The moduli stack $\cMpot$,
after removing the generic stabilizer group $\bG_m$,
is birational to the GIT quotient $\Mpotbar$.
This in turn is isomorphic
to the weighted projective space
$\bP(1,2,3,4)$.
 \item \label{it:finite}
There is a natural morphism $\cMpot \to \cMqui$,
which induces a finite morphism
$\Mpotbar \to \Mquibar$
of degree 4.
\end{enumerate}
\end{theorem}

The moduli space $\cMqui$
also appears in \cite{1402.3768}
as the moduli space of 4-qubit states.
\pref{th:main} (\pref{it:quintuple}) is
a consequence of known results
\cite{MR2039690,MR2409702},
and has little claim in originality.

This paper is organized as follows:
In \pref{sc:ncquadric},
we summarize results from
\cite{MR2836401} on noncommutative quadric surfaces.
\pref{th:main} (\pref{it:quintuple}) is proved
in \pref{sc:quintuple}, and
\pref{th:main} (\pref{it:quadruple}) is proved
in \pref{sc:quadruple}.
The proofs of Theorems \ref{th:main} (\pref{it:potential})
and \ref{th:main} (\pref{it:finite}) are given
in \pref{sc:potential}.
In \pref{sc:conifold},
we introduce the notion of noncommutative conifolds,
and discuss their relation with potentials.
In \pref{sc:DT},
we show that the motivic Donaldson-Thomas invariants
depend on the potential.
%
%

\emph{Acknowledgment}.
S.~O. is supported by JSPS Grant-in-Aid for Young Scientists No.~25800017.
K.~U. is supported by JSPS Grant-in-Aid for Young Scientists No.~24740043.
A part of this work is done
while K.~U. is visiting Korea Institute for Advanced Study,
whose hospitality and nice working environment
is gratefully acknowledged.

\section{Noncommutative quadric surfaces}
 \label{sc:ncquadric}

See \cite[Section 2]{MR2836401} for more details
of the contents of this section.
A \emph{$\bZ$-algebra} is an algebra with a direct sum decomposition
\begin{align}
 A = \bigoplus_{i,j \in \bZ} A_{ij}
\end{align}
such that the product satisfies
\begin{gather}
 A_{ij} A_{kl}= 0 \text{ for } j \ne k, \text{ and } \\
 A_{ij} A_{jk} \subset A_{ik}. \phantom{\text{ and }}
\end{gather}
We assume that $A_{ii}$ for any $i \in \bZ$
has an element $e_i$, called the \emph{local unit},
which satisfies
\begin{align}
 f e_i &= f \text{ for any } f \in A_{ki}, \text{ and} \\
 e_i g &= g \text{ for any } g \in A_{ij}.
\end{align}
A $\bZ$-algebra $A$ is \emph{positively graded} if
\begin{align}
 A_{ij} = 0 \text{ for } i > j.
\end{align}
A positively graded $\bZ$-algebra $A$ is \emph{connected} if
\begin{align}
 \dim A_{ij} < \infty \text{ and }
 A_{ii} = \bC \, e_i \text{ for any } i, j \in \bZ.
\end{align}
A \emph{graded $A$-module}
is a right $A$-module of the form
$M = \bigoplus_{i \in \bZ} M_i$
such that
\begin{gather}
 M_i A_{kl}= 0 \text{ for } i \ne k, \text{ and } \\
 M_i A_{ij} \subset M_j. \phantom{\text{ and }}
\end{gather}
A connected $\bZ$-algebra $A$ is \emph{AS-regular} if
\begin{align}
 &\dim A_{ij} \text{ is bounded by a polynomial in } j-i,\\
 &\text{ the projective dimension of} \  S_i
 \text{ is bounded by a constant independent of } i,
  \text{ and}\\
 &\textstyle{\sum_{j,k}} \dim \Ext_{\Gr A}^j(S_k, P_i) = 1
 \text{ for any } i \in \bZ,
\end{align}
where $P_i = e_i A$ is a projective $A$-module,
and $S_i = \bC e_i$ is a simple $A$-module.
An AS-regular $\bZ$-algebra $A$ is
a \emph{3-dimensional cubic AS-regular $\bZ$-algebra}
if the minimal resolution of $S_i$ is of the form
\begin{align}
 0 \to P_{i+4} \to P_{i+3}^2 \to P_{i+1}^2 \to P_i \to S_i \to 0.
\end{align}
A graded $A$-module $M$ is \emph{positively bounded}
if $M_i = 0$ for sufficiently large $i$.
A graded $A$-module is a \emph{torsion module}
if it can be described as the colimit
of a sequence of positively bounded $A$-modules.
The quotient abelian category
of the abelian category $\Gr A$
of graded $A$-modules
by the Serre subcategory $\Tor A$
consisting of torsion modules
will be denoted by
\begin{align}
 \Qgr A = \Gr A / \Tor A.
\end{align}
A \emph{noncommutative quadric surface} is
an abelian category of the form $\Qgr A$
for a 3-dimensional cubic AS-regular $\bZ$-algebra
\cite[Definition 3.2]{MR2836401}.

A 3-dimensional cubic AS-regular $\bZ$-algebra is \emph{linear}
if it is of the form
\begin{align}
 A &= \bigoplus_{i,j \in \bZ} \Hom
  \lb \cO_{\bP^1\times\bP^1}(-j),
    \cO_{\bP^1\times\bP^1}(-i) \rb
\end{align}
where
\begin{align}
 \cO_{\bP^1\times\bP^1}(n)
  &=
\begin{cases}
 \cO_{\bP^1}(k) \boxtimes \cO_{\bP^1}(k) & n = 2k, \\
 \cO_{\bP^1}(k) \boxtimes \cO_{\bP^1}(k+1) & n = 2k+1.
\end{cases}
\end{align}
All other 3-dimensional cubic AS-regular $\bZ$-algebra
will be called \emph{elliptic}.

A \emph{quintuple} $(V_0,V_1,V_2,V_3,W)$ consists of
\begin{align}
&\text{two-dimensional vector spaces
$V_0$, $V_1$, $V_2$, $V_3$, and}\\
&\text{a one-dimensional subspace $W$ of }
V_0 \otimes V_1 \otimes V_2 \otimes V_3.
\end{align}
A quintuple is \emph{geometric} if
a basis $w$ of $W$ satisfies
\begin{align}
 \la \phi_j \otimes \phi_{j+1}, \, w \ra \ne 0
  \in V_{j+2} \otimes V_{j+3}
\end{align}
for any $j = 0,1,2,3$ and any non-zero
$\phi_j\in V_{j}^{\vee}$, $\phi_{j+1}\in V_{j+1}^{\vee}$,
where indices are taken modulo 4
\cite[Definition 4.7]{MR2836401}.
A geometric quintuple is \emph{linear}
if it can be written as
\begin{align}
 w = x_0 x_1 y_2 y_3 - y_0 x_1 x_2 y_3
  - x_0 y_1 y_2 x_3 + y_0 y_1 x_2 x_3
\end{align}
for a suitable choice of a basis $\{ x_i, y_i \}$
of $V_i$ for $i=0,1,2,3$ (see \cite[Lemma 4.6]{MR2836401}).
All other geometric quintuples are called
\emph{elliptic}.
Two quintuples
$(V_0,V_1,V_2,V_3,W)$ and $(V_0',V_1',V_2',V_3',W')$
are isomorphic
if there are linear isomorphisms $V_i \cong V_i'$
for $i=0,1,2,3$
sending
$
 W \subset V_0 \otimes V_1 \otimes V_2 \otimes V_3
$
to
$
 W' \subset V_0' \otimes V_1' \otimes V_2' \otimes V_3'.
$

A globally generated line bundle $\scrL$ on a scheme $C$
defines a morphism
$\phi_\scrL : C \to \bP(H^0(\scrL)^{\vee})$,
whose direct product will be denoted by
\begin{align}
 \phi_{\scrL,\scrL'} \colon
  C \to \bP(H^0(\scrL)^{\vee}) \times \bP(H^0(\scrL')^{\vee}).
\end{align}
An \emph{admissible quadruple}
$
 (C, \scrL_0, \scrL_1, \scrL_2)
$
consists of a curve $C$ and line bundles $\scrL_0$,
$\scrL_1$, $\scrL_2$ on $C$ of degree 2 such that
\begin{align}
&h^0(\scrL_i) = 2 \text{ for } i=0,1,2, \\
&C \text{ is embedded
as a divisor of bidegree $(2,2)$ in $\bP^1\times \bP^1$
by both $\phi_{\scrL_0, \scrL_1}$ and $\phi_{\scrL_1, \scrL_2}$,} \\
&\deg (\scrL_0|_E)=\deg(\scrL_2|_E)
\text{ for every irreducible component $E$ of $C$, and} \\
&\scrL_0\not\cong \scrL_2.
\end{align}
Two admissible quadruples
$
 (C, \scrL_0, \scrL_1, \scrL_2)
$
and
$
 (C', \scrL_0', \scrL_1', \scrL_2')
$
are \emph{isomorphic}
if there is an isomorphism
$\varphi : C \to C'$ of curves
such that $\varphi^*\scrL_i' \cong \scrL_i$
for $i=0,1,2$.

There is a bijective correspondence
among isomorphisms classes of
\begin{align}
 &\text{elliptic 3-dimensional
cubic AS-regular $\bZ$-algebras},
 \label{it:elliptic_AS-regular} \\
 &\text{elliptic quintuples, and}
 \label{it:elliptic_quintuple} \\
 &\text{admissible quadruples}, \hspace{110mm}
 \label{it:admissible_quadruple}
\end{align}
given as follows:
\begin{itemize}
 \item
\cite[\S 4.2]{MR2836401} :
For a 3-dimensional cubic AS-regular $\bZ$-algebra
$
 A = \bigoplus_{i\le j} A_{ij},
$
set $V_i=A_{i, i+1}$ for $i=0,1,2,3$ and 
$
 R_i=\ker (V_i\otimes V_{i+1}\otimes V_{i+2}\to A_{i,i+3})
$
for $i=0,1$. Then
$
 W_0=R_0\otimes V_3\cap V_0\otimes R_1
$
is one-dimensional, and
$(V_0, V_1, V_2, V_3, W_0)$ gives a geometric quintuple.
 \item
\cite[\S 4.4]{MR2836401} :
For a geometric quintuple
$
 Q=(V_0, V_1, V_2, V_3, W_0=\bk w),
$
choose bases $\{ x_i,y_i \}$ of $V_i$
and write $w=f x_3+g y_3$.
Let
$
 C \subset
\bP(V_0)\times \bP(V_1)\times \bP(V_2)
$
be the variety defined by the equations $\{f,g\}$, and
set $\scrL_i=p^\ast_i(\scrO(1))$ for $i=0,1,2$,
where $p_i \colon C \to \bP(V_i)$ are the projections.
The geometricity of $Q$ implies that
$\phi_{\scrL_0, \scrL_1} = (p_0,p_1)$ and
$\phi_{\scrL_1, \scrL_2} = (p_1,p_2)$ are closed embeddings,
and one obtains a quadruple
$(C,\scrL_0,\scrL_1,\scrL_2)$.
\item
\cite[\S 4.5]{MR2836401} :
For an admissible quadruple $(C,\scrL_0,\scrL_1,\scrL_2)$,
construct a helix $(\scrL_i)_{i \in \bZ}$ by
\[
\scrL_i\otimes \scrL^{-1}_{i+1}\otimes
  \scrL^{-1}_{i+2}\otimes \scrL_{i+3}=\scrO_C.
  \]
Then the $\bZ$-algebra
generated by $A_{i, i+1}=V_i=H^0(C, \scrL_i)$
with the relations
$R_i=\ker (V_i\otimes V_{i+1}\otimes V_{i+2}\to A_{i,i+3})$
is an elliptic 3-dimensional cubic AS-regular $\bZ$-algebra.
\end{itemize}

\section{Moduli of quintuples}
 \label{sc:quintuple}

The invariant ring of
$\bC[V_0 \otimes V_1 \otimes V_2 \otimes V_3]$
with respect to the action of
$
 \SL(V_0) \times \SL(V_1) \times \SL(V_2) \times \SL(V_3)
$
is given in \cite{MR2039690}.
In this section,
we give a description of this invariant ring
along the lines of \cite{MR1910235,MR2409702}
and prove \pref{th:main}.\ref{it:quintuple}.

Let
\begin{align}
 \omega( x_1 \bse_1 + x_2 \bse_2, y_1 \bse_1 + y_2 \bse_2)
  &= x_1 y_2 - x_2 y_1
\end{align}
be a skew-symmetric bilinear form on $\bC^2$.
The action of $\SL_2(\bC) \times \SL_2(\bC)$
on $\bC^2 \otimes \bC^2$
defined by
\begin{align}
 (g, h) \cdot u \otimes v = (gu) \otimes (h v).
\end{align}
preserves the non-degenerate symmetric bilinear form
\begin{align}
 \lb v \otimes w, \, x \otimes y \rb
  &= \omega(v, x) \omega(w, y),
\end{align}
whose Gram matrix with respect to the basis
$
\{
 \bse_1 \otimes \bse_1,
 \bse_1 \otimes \bse_2,
 \bse_2 \otimes \bse_1,
 \bse_2 \otimes \bse_2
\}
$
is given by
\begin{align}
 J =
\begin{pmatrix}
0 & 0 & 0 & 1\\
0 & 0 & -1 & 0\\
0 & -1 & 0 & 0\\
1 & 0 & 0 & 0\\
\end{pmatrix}.
\end{align}
An orthonormal basis with respect to $J$ is given by
\begin{align}
T:=
\frac{1}{\sqrt{2}}
\begin{pmatrix}
1 & 0 & 0 & 1\\
0 & \sqrt{-1} & \sqrt{-1} & 0\\
0 & -1 & 1 & 0\\
\sqrt{-1} & 0 & 0 & -\sqrt{-1}\\
\end{pmatrix},
\end{align}
so that one obtains a homomorphism
\begin{align} \label{eq:accidental}
\begin{array}{cccc}
 \varphi : & \SL_2(\bC) \times \SL_2(\bC) & \to & \SO_4(\bC) \\
 & \vin & & \vin \\
 & (g_1,g_2) & \mapsto & T^T(g_1\otimes g_2)T
\end{array}
\end{align}
of algebraic groups.
The induced morphism
\begin{align}
 (d \varphi)_e : \fraks \frakl_2(\bC) \oplus \fraks \frakl_2(\bC)
  \to \fraks \frako(\bC)
\end{align}
of Lie algebras is an isomorphism 
called the \emph{accidental isomorphism}.
Hence the morphism $\varphi$ is \'{e}tale surjective,
and one can easily see that
\begin{align}
 \Ker \varphi = \{\pm (I, I)\} \in \SL_2(\bC) \times \SL_2(\bC).
\end{align}
One can identify the ring $\bC[V_0 \otimes V_1 \otimes V_2 \otimes V_3]$
with the ring $\bC[U \otimes U']$
equipped with the action of $\SO(U) \times \SO(U')$
through the homomorphisms
$\SL(V_0) \times \SL(V_1) \to \SO(U)$ and
$\SL(V_2) \times \SL(V_3) \to \SO(U')$,
in such a way that the invariant rings are isomorphic;
\begin{align}
 \bC[V_0 \otimes V_1 \otimes V_2 \otimes V_3]^{
  \SL(V_0) \times \SL(V_1) \times \SL(V_2) \times \SL(V_3)}
   \cong
 \bC[U \otimes U']^{\SO(U) \times \SO(U')}.
\end{align}
The action of $\SO(U) \times \SO(U')$ on $U \otimes U'$
can be identified with the action of $\SO(U) \times \SO(U)$
on $\End U$ given by
\begin{align}
 (g, h) \cdot X = g X h^T.
\end{align}
Singular value decomposition allows us to turn
any $n \times n$ matrix into a diagonal matrix
under this action.
The entries $x_1,\ldots,x_n$
of the resulting diagonal matrix is unique up to
\begin{itemize}
 \item
permutations of $x_1,\ldots,x_n$, and
 \item
sign changes $x_i\mapsto \epsilon_ix_i$,
where $\epsilon_i=\pm1$ and
$\epsilon_1\cdots\epsilon_n=1$.
\end{itemize}
As a consequence,
one has an isomorphism
\begin{align}\label{singular_values}
 \bC[\End U]^{\SO(U)\times \SO(U)}
  &\cong \bC[x_1,x_2,x_3,x_4]^\mathsf{W} \\
  &= \bC[f_2, f_4, g_4, f_6],
\end{align}
where $\mathsf{W}=\frakS_4\ltimes (\bZ/2\bZ)^3$
acts on $\bC^4$ by permutations and sign changes,
and
\begin{align}
 f_{2d} &= \sum_{i=1}^{4}x_i^{2d}, \\
 g_4 &= x_1x_2x_3x_4.
\end{align}
The corresponding elements
of $\bC[\End U]^{\SO(U)\times \SO(U)}$ are
\begin{align}
 f_{2d}(X) &= \Trace{((X^TX)^d)}, \\
 g_4 &= \det(X).
\end{align}
It follows that
\begin{align}
 &\text{an element $X \in \End U$ is stable
if and only if $g_4(X) \ne 0$}, \\
 &\text{an element $X \in \End U$ is unstable
if and only if $X^T X$ is nilpotent, and}\\
 &\text{the GIT quotient is given by }
 \Proj \bC[f_2,f_4,g_4,f_6]
  = \bP(2,4,4,6).
  \hspace{40mm}
\end{align}

\section{Elliptic quadruples}
 \label{sc:quadruple}

The principal congruence subgroup
\begin{align}
 \Gamma(2)=\Ker(\SL_2(\bZ)\to\SL_2(\bF_2))
\end{align}
of the modular group $\SL_2(\bZ)$
acts on the upper half plane $\bH$ by
\begin{align}
\left(\gamma=
\begin{pmatrix}
a & b\\
c & d\\
\end{pmatrix}, \tau\right)
\mapsto
\frac{a\tau+b}{c\tau+d}.
\end{align}
The generic stabilizer of this action is $\la \pm I_2 \ra$,
and the induced action of $\Gamma(2)/\langle\pm I_2\rangle$ is free.
The quotient $X'(2)=\bH/\Gamma(2)$
can be compactified to the modular curve
by adding three cusps.
The modular curve
$
 X(2) \cong \Proj \bC[\lambda_0, \lambda_1]
$
is known as the $\lambda$-line,
which parameterizes the family
\begin{align}
 E_{\lambda}=
  \lc [X:Y:Z] \in \bP(1,1,2) \relmid
   Z^2=XY(X-Y)(\lambda_1X-\lambda_0Y) \rc
\end{align}
of elliptic curves
equipped with level-2 structures.

Consider the action of the group $\Gamma(2)\ltimes \bZ^2$ on $\bH\times \bC$ defined by
\begin{align}
 (\gamma,m,n) \colon (\tau, z) \mapsto
  \lb \frac{a\tau+b}{c\tau+d}, \frac{z+m\tau+n}{c\tau+d} \rb.
\end{align}
The action of the subgroup $\bZ^2$ is free, and
the action of the quotient group $\Gamma(2)$
on $\bH\times\bC/\bZ^2$ is free
outside of the four sections
corresponding to the four 2-division points.
Let $S'(2)$ be the branched double cover of the quotient
$\bH\times\bC/\Gamma(2)\ltimes \bZ^2$ along the images of those four sections.
The natural projection
\begin{align}
 S'(2) \to X'(2), \quad (\tau, z) \mapsto \tau
\end{align}
is an elliptic fibration,
which can be identified with the family
\begin{align}
V(Z^2-XY(X-Y)(\lambda_1X-\lambda_0Y))\subset \bP(1,1,2)_{X:Y:Z}\times\bA^1_{\lambda}\to
\bA^1_{\lambda}.
\end{align}
The origin of the elliptic curve is given by $[1:0:0]$,
and other 2-torsion points are given by
$[1:1:0]$, $[\lambda_0:\lambda_1:0]$, and $[0:1:0]$.

Let $\Sbar(2)$ be the natural compactification
\begin{align}
 V(Z^2-XY(X-Y)(\lambda_1X-\lambda_0Y))
  \subset \bP(1,1,2)_{X:Y:Z}\times\bP^1_{\lambda}
\end{align}
of $S'(2)$.
The natural morphism to $\bP^1_{\lambda}\cong X(2)$ has
three singular fibers of type $I_1$ at the cusps $\lambda=0,1,\infty$.
The surface $\Sbar(2)$ itself has three $A_1$-singularities
at the singular points of the singular fibers,
and we write its minimal resolution as $S(2)$.

\begin{proposition}\label{modular_surface_of_level_two}
The elliptic fibration $S(2) \to X(2)$ has
three singular fibers of type $I_2$
at $\lambda=0,1,\infty$,
and the natural action of $\SL(2, \bF_2)\ltimes (\bZ/2\bZ)^2$
on $S'(2)$
extends to $S(2)$.
\end{proposition}
\begin{proof}
In the coordinates of $\Sbar(2)$, the three translations by the points $(1:1:0),
(\lambda_0:\lambda_1:0), (0:1:0)$ of order two can be written respectively as follows:
\begin{align}
(X:Y:Z) &\mapsto (\lambda(X-Y): X-\lambda Y: \lambda (\lambda -1)Z)\\
(X:Y:Z) &\mapsto (-X+\lambda Y: -X+Y: (\lambda -1)Z)\\
(X:Y:Z) &\mapsto (\lambda Y: X: \lambda Z)
\end{align}
By direct calculations, we can check that these are merely birational maps on $\Sbar(2)$ but
are genuine automorphisms on $S(2)$.
\end{proof}

Consider the three-fold $S(2)\times_{X(2)}S(2)$,
which is a compacification
of the $\bZ/2\bZ\times \bZ/2\bZ$-covering of the quotient
$
 \bH \times \bC \times \bC / \Gamma(2) \ltimes \bZ^4
$
by the action
\begin{align}
 (\gamma, m_1, n_1, m_2, n_2) \colon
  (\tau, z_1, z_2) \mapsto
   \lb \gamma(\tau), \frac{z_1+m_1\tau+n_1}{c\tau+d},
     \frac{z_2+m_2\tau+n_2}{c\tau+d} \rb.
\end{align}
To a point
$
 p=(\tau, z_1, z_2)\in S'(2)\times_{X'(2)}S'(2),
$
one can associate the quadruple
\begin{align} \label{eq:ell_quad}
 (E_\tau, \scrO(2o), \scrO(2z_1), \scrO(2z_2)),
\end{align}
where
$
 o = [0] \in E_\tau = \bC / \bZ \oplus \bZ \tau
$
is the origin.
The quadruple \eqref{eq:ell_quad} is admissible
if $2 z_2 \not \sim 2 o$.
Two elliptic quadruples
$
 (E_\tau, \scrO(2o), \scrO(2z_1), \scrO(2z_2))
$
and
$
 (E_{\tau'}, \scrO(2o), \scrO(2z_1'), \scrO(2z_2'))
$
are isomorphic if
\begin{enumerate}
 \item
the elliptic curves $E_\tau$ and $E_{\tau'}$ are isomorphic,
i.e., $\tau' = \gamma(\tau)$ for some $\gamma \in \SL_2(\bZ)$, and
 \item
$2 z_i \sim 2 z_i'$ for $i = 1,2,$,
i.e., $z_i$ and $z_i'$ are related by 2-torsion translation.
\end{enumerate}
Together, they form the group
\begin{align}
 G=\SL_2(\bF_2) \ltimes (\bZ/2\bZ)^4
\end{align}
acting on $S'(2) \times_{X'(2)} S'(2)$, and
one obtains the following:

\begin{proposition} \label{pr:quadruple}
Two points
$
 p, q \in S'(2)\times_{X'(2)}S'(2)
$
belong to the same orbit of the action of $G$
if and only if the associated quadruples are isomorphic.
\end{proposition}

\pref{th:main} (\pref{it:quadruple}) follows from
\pref{pr:quadruple}
and the equivalence
between \eqref{it:elliptic_quintuple}
and \eqref{it:admissible_quadruple}.

\section{Potentials}
 \label{sc:potential}

One can use the homomorphism \eqref{eq:accidental}
to identify the two rings
$
 \bC \ld \Sym^2(V_0 \otimes V_1) \rd
$
and $\bC \ld \Sym^2(U) \rd$
in such a way that the invariant subrings are isomorphic;
\begin{align} \label{eq:pot2}
 \bC \ld \Sym^2(V_0 \otimes V_1) \rd^{\SL(V_0) \times \SL(V_1)}
  \cong \bC \ld \Sym^2(U) \rd^{\SO(U)}.
\end{align}
Any symmetric matrix $X \in \Sym^2(U)$ can be diagonalized
by the conjugate action
\begin{align}
 g \cdot X = g X g^T
\end{align}
of $\SO(U)$.
The entries $x_1,\ldots,x_4$
of the resulting diagonal matrix is unique up to
permutations.
As a consequence,
one has an isomorphism
\begin{align}
 \bC \ld \Sym^2(U) \rd^{\SO(U)}
  \cong \bC[x_1,x_2,x_3,x_4]^{\frakS_4}
  = \bC[f_1, f_2, f_3, f_4],
\end{align}
where 
\begin{align}
 f_{d} &= \sum_{i=1}^{4}x_i^{d}.
\end{align}
The corresponding elements
of $\bC \ld \End U \rd^{\SO(U)}$ are
\begin{align}
 f_{d}(X) &= \tr \lb X^d \rb.
\end{align}
It follows that
\begin{align}
 &\text{an element $X \in \Sym^2(U)$ is unstable
if and only if $X$ is nilpotent, and}\\
 &\text{the GIT quotient is given by }
 \Proj \bC[f_1,f_2,f_3,f_4]
  = \bP(1,2,3,4).
  \hspace{40mm}
\end{align}

With a potential
$
 \Phi \in \Sym^2(V_0 \otimes V_1)
  \subset (V_0 \otimes V_1) \otimes (V_0 \otimes V_1),
$
one can associate the quintuple
$(V_0, V_1, V_0, V_1, \bC \Phi)$.
This gives a morphism
\begin{align}
 F \colon \cMpot \to \cMqui
\end{align}
of stacks,
which induces a rational map
of the corresponding GIT quotients.
This rational map can be identified with the natural projection
from
$
 \Mpotbar \cong \bP^3 / \frakS_3
$
to
$
 \Mpotbar \cong \bP^3 / (\frakS_3 \ltimes (\bZ/2\bZ)^3),
$
and hence a posteriori is a morphism.
Since the action
$
 [x_1:x_2:x_3:x_4]
  \mapsto [-x_1:-x_2:-x_3:-x_4]
$
is trivial on $\bP^3$,
the degree of this morphism coincides
with the cardinality of the group
$
 (\bZ/2\bZ)^3/(\bZ/2\bZ)\cong(\bZ/2\bZ)^2,
$
which is four.
The invariants $g_4$ and $f_6$
can be described in $\bC[f_1,f_2,f_3,f_4]$ as
\begin{align}
 g_4 &= \frac{f_1^4}{24}-\frac{f_1^2 f_2}{4}+\frac{f_1 f_3}{3}
 +\frac{f_2^2}{8}-\frac{f_4}{4},\\
 f_6 &= -\frac{f_1^6}{24}+\frac{3 f_1^4 f_2}{8}-\frac{2 f_1^3 f_3}{3}
       -\frac{3 f_1^2 f_2^2}{8}+\frac{3 f_1^2 f_4}{4}-\frac{f_2^3}{8}
       +\frac{3 f_2 f_4}{4}+\frac{f_3^2}{3}.
\end{align}

\section{Noncommutative conifolds}
 \label{sc:conifold}
 
A {\em quiver} $(Q_0, Q_1, s, t)$ consists of
a set $Q_0$ of vertices,
a set $Q_1$ of arrows, and
two maps $s, t: Q_1 \to Q_0$ from $Q_1$ to $Q_0$.
The vertices $s(a)$ and $t(a)$
are called the {\em source}
and the {\em target} of the arrow $a$.
A {\em path} on a quiver
is an ordered set of arrows
$(a_n, a_{n-1}, \dots, a_{1})$
such that $s(a_{k+1}) = t(a_k)$
for $k=1, \dots, n-1$.
We also allow for a path $e_i$ of length zero,
starting and ending at the same vertex $i \in Q_0$.
The {\em path algebra} $\bC Q$
of a quiver $Q$ is the algebra
spanned by the set of paths
as a vector space,
and the multiplication is defined
by the concatenation of paths;
$$
 (b_m, \dots, b_1) \cdot (a_n, \dots, a_1)
  = \begin{cases}
     (b_m, \dots, b_1, a_n, \dots, a_1) & s(b_1) = t(a_n), \\
      0 & \text{otherwise}.
    \end{cases}
$$
A \emph{cyclic path} is a path $(a_n, \ldots, a_1)$
such that $s(a_1) = t(a_n)$.
There is a cyclic group action
$(a_n, \ldots, a_1) \mapsto (a_{n-1}, \ldots, a_1, a_n)$
on the set of cyclic paths.
A \emph{potential}
is an element $\Phi$
of the invariant subspace $(\bC Q)^\cycl$
of the subspace of $\bC Q$
spanned by cyclic paths.
The partial derivative of a path
by an arrow $a$ is defined by
\begin{align}
 \frac{\partial (a_n, \ldots, a_1)}{\partial b} =
\begin{cases}
 (a_{n-1}, \ldots, a_1) & a_n = b, \\
 0 & \text{otherwise}.
\end{cases}
\end{align}
The \emph{Jacobi algebra}
of a quiver $(Q, \Phi)$ with potential is the quotient
$
 \bC Q / (\partial \Phi)
$
of the path algebra $\bC Q$
by the two-sided ideal
\begin{align}
 (\partial \Phi) &=
  \lb
   \frac{\partial \Phi}{\partial a}
  \rb_{a \in Q_1}
\end{align}
generated by the partial derivatives.
The path algebra of a quiver has a natural grading
coming from the length of paths.
The Jacobi algebra $\bC Q/(\partial \Phi)$
inherits this grading
if the potential $\Phi$ is homogeneous.
 
An algebra $A$ is \emph{homologically smooth}
if $A$ has finite projective dimension
as an $A$-bimodule
(i.e. as a module over $A^e = A \otimes_{\bC} A^\op$).
A homologically smooth algebra is a
\emph{Calabi-Yau algebra of dimension 3}
if there is an isomorphism
$
 f \colon A \simto A^![3]
$
such that $f=f^![3]$,
where $A^! = \RHom_{A^e}(A, A^e)$
is the inverse of the rigid dualizing complex
\cite[Definition 3.2.3]{Ginzburg_CYA}.

Let $A$ be a flat deformation
of the noncommutative crepant resolution
$A_0 = \bC Q/(\partial \Phi_0)$
of the conifold,
where $Q$ is the quiver in \pref{fg:conifold_quiver}
and $\Phi_0$ is the potential in \eqref{eq:conifold_potential}.
We assume that $A$ is a graded Calabi-Yau algebra
of dimension 3,
so that it can be described as the Jacobi algebra
$\bC Q/(\partial \Phi)$
of a quiver $(Q, \Phi)$ with potential
by \cite[Theorem 3.1]{Bocklandt_GCYA}.
The flatness of the deformation
implies the invariance of the Hilbert series
of the graded ring.
It follows that the degree of the potential $\Phi$
must be preserved under deformation.

Let $\Qtilde$ be the quiver in \pref{fg:double_quiver}
obtained as the the double cover of the quiver $Q$.
The pull-back of the potential $\Phi$
by the natural homomorphism $\bC \Qtilde \to \bC Q$
gives a potential $\Phitilde$ of the quiver $\Qtilde$.
The quotient
of the Jacobi algebra $\bC \Qtilde / (\partial \Phitilde)$
by the two-sided ideal generated by $b_1'$ and $b_2'$
gives the path algebra of the quiver in \pref{fg:P1P1_quiver}
with two relations among
eight paths from $v_{0,0}$ to $v_{1,1}$.
These two relations can be regarded
either as a two-dimensional subspace
of $V_0 \otimes V_1 \otimes V_0$
or a one-dimensional subspace
of $V_0 \otimes V_1 \otimes V_0 \otimes V_1$.
The latter coincides
with the $V_0 \otimes V_1 \otimes V_0 \otimes V_1$ component
of the potential $\Phitilde$,
which we write as $\Phi$ by abuse of notation.
In this way,
one obtains a quintuple
$(V_0, V_1, V_0, V_1, \bC \Phi)$
from a potential $\Phi$.
The potential $\Phi$ is said to be \emph{geometric}
if the associated quintuple is geometric.

It follows from the construction of the $\bZ$-algebra
$A = \bigoplus_{i,j \in \bZ} A_{ij}$
from a quintuple that 
the Jacobi algebra $\bC Q/(\partial \Phi)$
is related to the $\bZ$-algebra by
\begin{align}
 e_i \bC Q/(\partial \Phi) e_j
  &=
\begin{cases}
 \bigoplus_{n=0}^\infty A_{0,2n} & i = j, \\
 \bigoplus_{n=0}^\infty A_{0,2n+1} & i \ne j.
\end{cases}
\end{align}
The flatness of the deformation
is equivalenct to the geometricity
of the potential
\cite[Lemma 4.9]{MR2836401}.
Geometricity of $\Phi$ implies that
the associated $\bZ$-algebra $A$ is
AS-regular of dimension 3
\cite[Theorem 4.31]{MR2836401},
which in turn implies that
$\bC Q/(\partial \Phi)$ is Calabi-Yau of dimension 3.
The subring $R=e_0 \bC Q/(\partial \Phi) e_0$
of the Jacobi algebra
of a geometric potential
will be called a \emph{noncommutative conifold}.

A \emph{categorical resolution}
{\cite[Definition 1.3]{1212.6170}
(cf.~also~\cite{MR2403307,MR2609187})
of a ring $R$
is a smooth cocomplete
compactly generated triangulated category $\scrT$
and an adjoint pair
$
 \pi^* : D(R) \rightleftharpoons \scrT : \pi_*
$
of functors
such that
\begin{enumerate}
 \item
the natural morphism
$
 \id_{D(R)} \to \pi_* \pi^*
$
of functors is an isomorphism,
 \item
both $\pi^*$ and $\pi_*$ commute
with arbitrary direct sum, and
 \item
$\pi_*(\scrT^c) \subset D^b \module (R)$.
\end{enumerate}
When $R$ is the subring
$e_0 \bC Q / (\partial \Phi) e_0$
of the Jacobi algebra $\bC Q / (\partial \Phi)$
of the conifold quiver
with a geometric potential $\Phi$,
the adjoint pair $\pi^* \vdash \pi_*$
of the functors
\begin{align}
 \pi_* : D(\bC Q / (\partial \Phi)) \to D(R), \quad
  M \mapsto M e_0
\end{align}
and
\begin{align}
 \pi^* : D(R) \to D(\bC Q / (\partial \Phi)), \quad
  N \mapsto N \otimes_R e_0 \bC Q / (\partial \Phi)
\end{align}
is a categorical resolution
of the noncommutative conifold $R$.

\section{Donaldson-Thomas invariants}
 \label{sc:DT}

\begin{figure}[t]
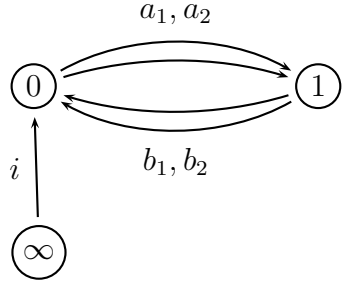

\centering
\vspace{8mm}
$
\begin{psmatrix}[mnode=circle]
0 & & 1 \\
\infty
\psset{nodesep=3pt,arrows=->,shortput=tablr}
\ncarc[arcangle=17]{1,1}{1,3}
\ncarc[arcangle=30]{1,1}{1,3}^{a_1, a_2}
\ncarc[arcangle=17]{1,3}{1,1}
\ncarc[arcangle=30]{1,3}{1,1}_{b_1, b_2}
\ncline{2,1}{1,1}<{i}
\end{psmatrix} 
$
\vspace{0mm}
\caption{The framed conifold quiver}
\label{fg:framed_conifold_quiver}
\end{figure}

Let $\Qtilde$ be the quiver
in \pref{fg:framed_conifold_quiver},
equipped with the potential $\Phitilde$
inherited from a potential $\Phi$
of the conifold quiver $Q$.
A \emph{framed representation}
of the conifold quiver
is a module over the Jacobi algebra
$\bC \Qtilde/(\partial \Phitilde)$,
i.e., a collection
$
 M=((V_0,V_1,V_\infty),(A_1,A_2,B_1,B_2,I))
$
of vector spaces $(V_0,V_1,V_\infty)$
and maps
\begin{align}
 A_i &\colon V_0 \to V_1, \quad i=1,2, \\
 B_i &\colon V_1 \to V_0, \quad i=1,2, \\
 I &\colon V_\infty \to V_0
\end{align}
satisfying the relation $(\partial \Phitilde)$.
The \emph{dimension vector}
of $M$ is given by
\begin{align}
 v
  = (v_0, v_1, v_\infty)
  = (\dim V_0, \dim V_1, \dim V_\infty).
\end{align}
A \emph{stability parameter} is an element
$
 \theta
  = (\theta_0, \theta_1, \theta_\infty)
$
of $\bR^3$.
The \emph{slope} of a representation
$
 M=((V_0,V_1,V_\infty),(A_1,A_2,B_1,B_2,I))
$
with respect to the stability parameter $\theta$
is defined by
\begin{align}
 \theta(M)
  = \frac{\theta_0 \dim V_0 + \theta_1 \dim V_1 + \theta_\infty \dim V_\infty}
   {\dim V_0 + \dim V_1 + \dim V_\infty}.
\end{align}
A representation $M$ is $\theta$-stable
if one has
\begin{align}
 \theta(M') < \theta(M).
\end{align}
for any non-zero proper submodule $M' \subset M$.
The moduli space
of framed $\theta$-stable representations
wit dimension $v$
will be denoted by $\cM_\theta(v)$.
The \emph{numerical DT invariant}
is defined as the weighted Euler characteristic
\eqref{eq:DT}
with respect to the Behrend function on $\cM_\theta(v)$.
We expect that the numerical DT invariant is independent
of the potential $\Phi$.

The numerical DT invariant has a motivic refinement,
which is defined in terms of \emph{virtual motives}
\cite{MR3032328}.
When the potential is the one
given in \eqref{eq:conifold_potential},
then $\cM_\theta(v)$ for
$v=(1,1,1)$ and $\theta = (-1,-1,2)$
is the resolved conifold
\begin{align} \label{eq:conifold_stack}
 Y = [\lb \bA^4 \setminus \{(0,0)\} \times \bA^2 \rb/\bCx]
  \cong \cSpec_{\bP^1}
   \lb \cSym^* \lb \cO_{\bP^1}(1) \oplus \cO_{\bP^1}(1) \rb \rb,
\end{align}
and the virtual motive is given by
\begin{align} \label{eq:motive1}
 [Y]_\vir = \bL^{-3/2} [Y],
\end{align}
where $\bL$ is the Lefschetz motive.
When the potential is given by
\begin{align}
 \Phi = \alpha a_1 b_1 a_1 b_1 + \beta a_1 b_2 a_1 b_2
  + \gamma a_2 b_1 a_2 b_1 + \delta a_2 b_2 a_2 b_2,
\end{align}
for general $\alpha$, $\beta$, $\gamma$ and $\delta$,
the moduli space $\cM_\theta(v)$
is the subscheme of $Y$ defined by
\begin{align}
 a_1(\alpha b_1^2+\beta b_2^2) &= 0, \\
 a_2(\gamma b_1^2+\delta b_2^2) &= 0, \\
 b_1(\alpha a_1^2+\beta a_2^2) &= 0, \\
 b_2(\gamma a_1^2+\delta a_2^2) &= 0.
\end{align}
The motivic DT invariant in this case is given by
\begin{align}
 -\bL^{-3/2} \lb [\Phi^{-1}(1)]-[\Phi^{-1}(0)] \rb
\end{align}
by \cite[Theorem B.1]{MR3032328},
where $\Phi$ is considered
as a map from $Y$ to $\bA^1$.
Being the difference of two subschemes,
this motive is clearly distinct
from \eqref{eq:motive1}.

\bibliographystyle{amsalpha}
\bibliography{bibs}

\def\cprime{$'$} \def\cprime{$'$}
\providecommand{\bysame}{\leavevmode\hbox to3em{\hrulefill}\thinspace}
\providecommand{\MR}{\relax\ifhmode\unskip\space\fi MR }
\providecommand{\MRhref}[2]{%
  \href{http://www.ams.org/mathscinet-getitem?mr=#1}{#2}
}
\providecommand{\href}[2]{#2}
\begin{thebibliography}{VDDMV02}

\bibitem[Ati58]{MR0095974}
M.~F. Atiyah, \emph{On analytic surfaces with double points}, Proc. Roy. Soc.
  London. Ser. A \textbf{247} (1958), 237--244. \MR{0095974 (20 \#2472)}

\bibitem[BBS13]{MR3032328}
Kai Behrend, Jim Bryan, and Bal{\'a}zs Szendr{\H{o}}i, \emph{Motivic degree
  zero {D}onaldson-{T}homas invariants}, Invent. Math. \textbf{192} (2013),
  no.~1, 111--160. \MR{3032328}

\bibitem[BCFKvS98]{Batyrev-Ciocan-Fontanine-Kim-van_Straten_CT}
Victor~V. Batyrev, Ionu{\c{t}} Ciocan-Fontanine, Bumsig Kim, and Duco van
  Straten, \emph{Conifold transitions and mirror symmetry for {C}alabi-{Y}au
  complete intersections in {G}rassmannians}, Nuclear Phys. B \textbf{514}
  (1998), no.~3, 640--666. \MR{MR1619529 (99m:14074)}

\bibitem[Beh09]{MR2600874}
Kai Behrend, \emph{Donaldson-{T}homas type invariants via microlocal geometry},
  Ann. of Math. (2) \textbf{170} (2009), no.~3, 1307--1338. \MR{2600874
  (2011d:14098)}

\bibitem[BO]{Bondal-Orlov_semiorthogonal}
A.~Bondal and D.~Orlov, \emph{Semiorthogonal decomposition for algebraic
  varieties}, arXiv:alg-geom/9506012.

\bibitem[Boc08]{Bocklandt_GCYA}
Raf Bocklandt, \emph{Graded {C}alabi {Y}au algebras of dimension 3}, J. Pure
  Appl. Algebra \textbf{212} (2008), no.~1, 14--32. \MR{MR2355031
  (2008h:16013)}

\bibitem[BP93]{MR1230966}
A.~I. Bondal and A.~E. Polishchuk, \emph{Homological properties of associative
  algebras: the method of helices}, Izv. Ross. Akad. Nauk Ser. Mat. \textbf{57}
  (1993), no.~2, 3--50. \MR{1230966 (94m:16011)}

\bibitem[Bri02]{Bridgeland_FDC}
Tom Bridgeland, \emph{Flops and derived categories}, Invent. Math. \textbf{147}
  (2002), no.~3, 613--632. \MR{MR1893007 (2003h:14027)}

\bibitem[DT98]{MR1634503}
S.~K. Donaldson and R.~P. Thomas, \emph{Gauge theory in higher dimensions}, The
  geometric universe ({O}xford, 1996), Oxford Univ. Press, Oxford, 1998,
  pp.~31--47. \MR{1634503 (2000a:57085)}

\bibitem[Fri86]{MR848512}
Robert Friedman, \emph{Simultaneous resolution of threefold double points},
  Math. Ann. \textbf{274} (1986), no.~4, 671--689. \MR{848512 (87k:32035)}

\bibitem[Gin06]{Ginzburg_CYA}
Victor Ginzburg, \emph{{C}alabi-{Y}au algebras}, math.AG/0612139, 2006.

\bibitem[KL]{1212.6170}
Alexander Kuznetsov and Valery Lunts, \emph{Categorical resolutions of
  irrational singularities}, arXiv:1212.6170.

\bibitem[Kuz08]{MR2403307}
Alexander Kuznetsov, \emph{Lefschetz decompositions and categorical resolutions
  of singularities}, Selecta Math. (N.S.) \textbf{13} (2008), no.~4, 661--696.
  \MR{2403307 (2009h:18022)}

\bibitem[KW99]{MR1666725}
Igor~R. Klebanov and Edward Witten, \emph{Superconformal field theory on
  threebranes at a {C}alabi-{Y}au singularity}, Nuclear Phys. B \textbf{536}
  (1999), no.~1-2, 199--218. \MR{1666725 (99k:81253)}

\bibitem[LT03]{MR2039690}
Jean-Gabriel Luque and Jean-Yves Thibon, \emph{Polynomial invariants of four
  qubits}, Phys. Rev. A (3) \textbf{67} (2003), no.~4, 042303, 5. \MR{2039690
  (2004k:81098)}

\bibitem[Lun10]{MR2609187}
Valery~A. Lunts, \emph{Categorical resolution of singularities}, J. Algebra
  \textbf{323} (2010), no.~10, 2977--3003. \MR{2609187 (2011d:18019)}

\bibitem[MMNS12]{MR2927365}
Andrew Morrison, Sergey Mozgovoy, Kentaro Nagao, and Bal{\'a}zs Szendr{\H{o}}i,
  \emph{Motivic {D}onaldson-{T}homas invariants of the conifold and the refined
  topological vertex}, Adv. Math. \textbf{230} (2012), no.~4-6, 2065--2093.
  \MR{2927365}

\bibitem[MN]{Morrison-Nagao}
A.~{Morrison} and K.~{Nagao}, \emph{{Motivic Donaldson-Thomas invariants of
  toric small crepant resolutions}}, arXiv:1110.5976.

\bibitem[MNOP06]{Maulik-Nekrasov-Okounkov-Pandharipande_I}
D.~Maulik, N.~Nekrasov, A.~Okounkov, and R.~Pandharipande,
  \emph{Gromov-{W}itten theory and {D}onaldson-{T}homas theory. {I}}, Compos.
  Math. \textbf{142} (2006), no.~5, 1263--1285. \MR{MR2264664 (2007i:14061)}

\bibitem[Nag12]{MR2999994}
Kentaro Nagao, \emph{Derived categories of small toric {C}alabi-{Y}au 3-folds
  and curve counting invariants}, Q. J. Math. \textbf{63} (2012), no.~4,
  965--1007. \MR{2999994}

\bibitem[NN11]{MR2836398}
Kentaro Nagao and Hiraku Nakajima, \emph{Counting invariant of perverse
  coherent sheaves and its wall-crossing}, Int. Math. Res. Not. IMRN (2011),
  no.~17, 3885--3938. \MR{2836398 (2012h:14141)}

\bibitem[OU]{1402.3768}
Shinnosuke Okawa and Kazushi Ueda, \emph{Quantum entanglement, {C}alabi-{Y}au
  manifolds, and noncommutative algebraic geometry}, arXiv:1402.3768.

\bibitem[OV00]{MR1765411}
Hirosi Ooguri and Cumrun Vafa, \emph{Knot invariants and topological strings},
  Nuclear Phys. B \textbf{577} (2000), no.~3, 419--438. \MR{1765411
  (2001i:81254)}

\bibitem[Rei87]{MR909231}
Miles Reid, \emph{The moduli space of {$3$}-folds with {$K=0$} may nevertheless
  be irreducible}, Math. Ann. \textbf{278} (1987), no.~1-4, 329--334.
  \MR{909231 (88h:32016)}

\bibitem[Sze08]{Szendroi_NCDT}
Bal{\'a}zs Szendr{\H{o}}i, \emph{Non-commutative {D}onaldson-{T}homas
  invariants and the conifold}, Geom. Topol. \textbf{12} (2008), no.~2,
  1171--1202. \MR{MR2403807 (2009e:14100)}

\bibitem[Tan79]{MR574272}
Sh{\^u}kichi Tanno, \emph{Geodesic flows on {$C\sb{L}$}-manifolds and
  {E}instein metrics on {$S\sp{3}\times S\sp{2}$}}, Minimal submanifolds and
  geodesics ({P}roc. {J}apan-{U}nited {S}tates {S}em., {T}okyo, 1977),
  North-Holland, Amsterdam, 1979, pp.~283--292. \MR{574272 (81g:58027)}

\bibitem[Tho00]{MR1818182}
R.~P. Thomas, \emph{A holomorphic {C}asson invariant for {C}alabi-{Y}au
  3-folds, and bundles on {$K3$} fibrations}, J. Differential Geom. \textbf{54}
  (2000), no.~2, 367--438. \MR{1818182 (2002b:14049)}

\bibitem[vdB04a]{Van_den_Bergh_NCR}
Michel van~den Bergh, \emph{Non-commutative crepant resolutions}, The legacy of
  Niels Henrik Abel, Springer, Berlin, 2004, pp.~749--770. \MR{MR2077594
  (2005e:14002)}

\bibitem[VdB04b]{Van_den_Bergh_TFNR}
Michel Van~den Bergh, \emph{Three-dimensional flops and noncommutative rings},
  Duke Math. J. \textbf{122} (2004), no.~3, 423--455. \MR{MR2057015
  (2005e:14023)}

\bibitem[VdB11]{MR2836401}
\bysame, \emph{Noncommutative quadrics}, International Mathematics Research
  Notices. IMRN (2011), no.~17, 3983--4026. \MR{MR2836401 (2012m:14004)}

\bibitem[VDDMV02]{MR1910235}
F.~Verstraete, J.~Dehaene, B.~De~Moor, and H.~Verschelde, \emph{Four qubits can
  be entangled in nine different ways}, Phys. Rev. A (3) \textbf{65} (2002),
  no.~5, part A, 052112, 5. \MR{1910235 (2003c:81033)}

\bibitem[Wal08]{MR2409702}
Nolan~R. Wallach, \emph{Quantum computing and entanglement for mathematicians},
  Representation theory and complex analysis, Lecture Notes in Math., vol.
  1931, Springer, Berlin, 2008, pp.~345--376. \MR{2409702 (2010g:81082)}

\bibitem[You09]{Young_CPP}
Ben Young, \emph{Computing a pyramid partition generating function with dimer
  shuffling}, J. Combin. Theory Ser. A \textbf{116} (2009), no.~2, 334--350.
  \MR{MR2475021 (2009k:05016)}

\bibitem[You10]{MR2643058}
Benjamin Young, \emph{Generating functions for colored 3{D} {Y}oung diagrams
  and the {D}onaldson-{T}homas invariants of orbifolds}, Duke Math. J.
  \textbf{152} (2010), no.~1, 115--153, With an appendix by Jim Bryan.
  \MR{2643058 (2011b:14125)}

\end{thebibliography}

\ \vspace{0mm} \\

\noindent
Shinnosuke Okawa

Department of Mathematics,
Graduate School of Science,
Osaka University,
Machikaneyama 1-1,
Toyonaka,
Osaka,
560-0043,
Japan.

{\em e-mail address}\ : \  okawa@math.sci.osaka-u.ac.jp
\ \vspace{0mm} \\

\ \vspace{0mm} \\

\noindent
Kazushi Ueda

Department of Mathematics,
Graduate School of Science,
Osaka University,
Machikaneyama 1-1,
Toyonaka,
Osaka,
560-0043,
Japan.

{\em e-mail address}\ : \  kazushi@math.sci.osaka-u.ac.jp
\ \vspace{0mm} \\

\end{document}